\documentclass{imsart}
\RequirePackage[OT1]{fontenc}
\RequirePackage{amsthm,amsmath,amssymb,amsthm}
\RequirePackage{natbib}
\RequirePackage[colorlinks,citecolor=blue,urlcolor=blue]{hyperref}
\usepackage{soul}
\usepackage{cancel}
\usepackage{ulem}

\startlocaldefs
\numberwithin{equation}{section}
\theoremstyle{plain}
\newtheorem{theorem}{Theorem}[section]

\newtheorem{corollary}{Corollary}[section]
\newtheorem{proposition}{Proposition}[section]

\newtheorem{lemma}{Lemma}[section]
\DeclareMathOperator{\sgn}{sgn}

\def \L{L_\lambda}

\def \en{\varepsilon_n}

\endlocaldefs

\begin{document}

\begin{frontmatter}
\title{Cube root weak convergence \\ of empirical estimators \\of a density level set}
\runtitle{Weak convergence for estimated level sets}

\begin{aug}
\author{\fnms{Philippe} \snm{Berthet}\ead[label=e1]{philippe.berthet@math.univ-toulouse.fr} \and
\fnms{John H.J.} \snm{Einmahl}\ead[label=e2]{j.h.j.einmahl@uvt.nl} 
\!\thanksref{t1}}
\runauthor{P. Berthet and  J.H.J. Einmahl}

\thankstext{t1}{John Einmahl holds the Arie Kapteyn Chair 2019--2022 and gratefully acknowledges the corresponding research support.}
\address{Institut de Math\'{e}matiques de Toulouse; UMR5219\\
Universit\'{e} de Toulouse; CNRS\\
UPS IMT, F-31062 Toulouse Cedex 9\\ France\\
\printead{e1}
}
\address{
Dept.\ of Econometrics and OR and CentER\\
Tilburg University\\
PO Box 90153, 5000 LE Tilburg\\
The Netherlands\\
\printead{e2}}

\end{aug}

\begin{abstract}

Given $n$ independent random vectors with common density $f$ on $\mathbb{R}^d$, we study the weak convergence of three empirical-measure based estimators of the convex $\lambda$-level set $L_\lambda$ of $f$, namely  the excess mass set, the minimum volume set and the  maximum probability set, all selected from a class of convex sets $\mathcal{A}$ that contains $L_\lambda$. Since these set-valued estimators approach $L_\lambda$, even the formulation of their weak convergence is non-standard. We identify the joint limiting distribution of the symmetric difference of $L_\lambda$ and each of the three estimators, at rate $n^{-1/3}$. It turns out that the minimum volume set and the maximum probability set estimators are asymptotically indistinguishable, whereas the excess mass set estimator exhibits ``richer" limit behavior. Arguments rely on the boundary local empirical process, its cylinder representation,  dimension-free concentration around the boundary of $L_\lambda$, and the set-valued argmax of a drifted Wiener process.

\end{abstract}

\begin{keyword}[class=MSC]
\kwd[Primary ]{62G05, 62G20}
\kwd[; secondary ]{60F05, 60F17}
\end{keyword}

\begin{keyword}
\kwd{Argmax drifted Wiener process, cube root asymptotics, density level set, excess mass, local empirical process, minimum volume set, set-valued estimator}
\end{keyword}

\tableofcontents
\end{frontmatter}

\section{Introduction}

\subsection{Three level set estimators}
Let $X_1, \dots, X_n  $, $n\in \mathbb{N}$, be  independent
and identically distributed random variables 
taking values in $\mathbb{\mathbb{R}}^{d}$, $d \in \mathbb{N}$, endowed with  Lebesgue measure $\mu$ and Borel sets $\mathcal{B}(\mathbb{R}^{d})$. Assume that the law $P$ of $X_1$ is absolutely continuous with respect to $\mu$ with continuous density $f$. We intend to establish  novel, non-standard weak limit theorems for three set-valued estimators of a convex level set of $f$, treated as random sets rather than estimated finite-dimensional  parameters.\smallskip

\noindent\textbf{Motivation.} Several classical problems in multivariate statistics involve set-valued estimators  based on $X_{1},\ldots ,X_{n}$.  For instance, in order to detect areas having high probability $P$, to localize modes or clusters, to test for multimodality, to find outliers, or to test for goodness-of-fit to a family of distributions. In particular, many approaches and procedures rely on $\lambda$-level sets $L_{\lambda}$ of the density $f$ ($\lambda>0$). The plug-in method consists of using the corresponding level set of some density estimator. Alternatively, estimators of $L_{\lambda}$ can be obtained by selecting a set in a  class $\mathcal{A}\subset\mathcal{B}(\mathbb{R}^{d})$ according to some optimization criterion applied  directly to the empirical measure of $X_{1},\ldots,X_{n}$. Here we avoid density estimation and follow the latter approach.  Note that maybe the most natural class of sets $\mathcal{A}$ is the class of all closed ellipsoids. We will consider the classical nonparametric M-estimators of $L_{\lambda}$ based on the  following three criteria:
\begin{itemize}
    \item excess mass,
    \item minimum volume, and
    \item maximum probability.
\end{itemize} 
In particular, the first two criteria have been studied in the literature extensively. The third one is also very natural, since it is a kind of inverse of the minimum volume approach. 

Seminal papers on the excess mass approach are \cite{MS91}, \cite{N91}, \cite{M92}, and \cite{P95}, and pioneering work on   the minimum volume approach can be found in \cite{ST80}, \cite{R85}, \cite{D92}, and  \cite{P97}. For the maximum probability approach we refer to \cite{P98}. For different, early approaches to the estimation of density level sets see
\cite{H87} and \cite{T97}, and for more recent work, see, e.g.,  \cite{C06}, \cite{CEH11}, and \cite{CGW17}. Statistical/machine learning approaches to  the aforementioned criteria, include \cite{CGS15} and \cite{SN06}.
As far as asymptotic theory is concerned, the  results in the literature regarding empirical estimators of the level sets study  rates of convergence towards the true level set for appropriately defined distances. Other types of results consider weak convergence  for estimators of the parameters of a parametrically defined level set. 

The main goal of this paper is to deal with the weak convergence of the three classical, competing \textit{set-valued estimators} of the level set $L_{\lambda}$ \textit{themselves} and look for  their  differences or similarities, 
jointly. Since these estimators approach $L_\lambda$, even the formulation of weak convergence is non-standard. Our main results are novel central limit theorems for the aforementioned three  empirical-measure based estimators of $L_{\lambda}$, which reveal their interesting  asymptotic behavior as random sets  and provide the distribution of their limiting sets, obtained  after cube-root-$n$  magnification. The proofs raised various challenges as indicated in Subsection \ref{chally} below.\smallskip

\noindent\textbf{Target level set.} Fix $\lambda>0$ throughout and assume that the level set
$$
L_{\lambda}=\left\{  x\in\mathbb{R}^{d}:f(x)\geq \lambda\right\}
$$ 
is a convex body, that is, it is convex, compact, and has non-empty interior, and that  $S_{\lambda}=\left\{  x\in \mathbb{R}^{d}:f(x)=\lambda\right\} $ coincides with its boundary: $S_{\lambda}=\partial L_{\lambda}$.
Note that $f>\lambda$ on $L_{\lambda} \setminus S_{\lambda}$ and $f<\lambda$ on $\mathbb{R}^{d}\setminus L_{\lambda}$. Hence,
$$
e_{\lambda}=p_{\lambda}-\lambda v_{\lambda}>0, \mbox{ with }
p_{\lambda}=P(L_{\lambda})\in (0,1)\mbox{ and } v_{\lambda}=\mu(L_{\lambda})\in (0,1/\lambda).
$$
We denote  the Hausdorff surface measure of $S_\lambda$ by $s_\lambda$ and have $s_\lambda\geq c_dv_\lambda^{1-1/d}>0$ by the isoperimetric inequality, with $c_d>0$. Let $\mathcal{A}%
\subset\mathcal{B}(\mathbb{R}^{d})$ be a class of closed, convex sets
with $L_{\lambda}\in \mathcal{A}$. Then we have
\begin{align*}
L_{\lambda}  & =\ \underset{A\in\mathcal{A}}{\arg\max}\left\{  P(A)-\lambda
\mu(A)\right\} \\
& =\ \underset{A\in\mathcal{A}}{\arg\min}\left\{  \mu(A):P(A)\geq 
p_{\lambda}\right\} \\
& =\ \underset{A\in\mathcal{A}}{\arg\max}\left\{  P(A):\mu(A)\leq 
v_{\lambda}\right\}  ,
\end{align*}
and the maximizing/minimizing set is unique. In other words, if $\lambda$ is known then $L_{\lambda}$ maximizes on
$\mathcal{A}$ the excess mass function $A\mapsto e_{\lambda}(A)=P(A)-\lambda
\mu(A)$, if $p_{\lambda}$ is known then $L_{\lambda}$ minimizes on
$\{A\in \mathcal{A}: P(A)\geq 
p_{\lambda}\}$ the volume function $A\mapsto\mu(A)$ and if $v_{\lambda}$ is
known then $L_{\lambda}$ maximizes on $\{A\in \mathcal{A}: 
\mu(A)\leq 
v_{\lambda}
\}$ the probability mass function
$A\mapsto P(A)$.\smallskip

\noindent\textbf{Empirical level sets.} Let $\delta_{x}$ denote the Dirac measure
at $x$. From the nonparametric viewpoint it is natural to estimate $P$ with the
empirical measure $P_{n}=n^{-1}%
{\textstyle\sum\nolimits_{i=1}^{n}}
\delta_{X_{i}} $ in the above argmax and argmin. 
To motivate a joint study, imagine that three
statisticians want to estimate the level set $L_{\lambda}$ by using the same sample $X_{1},...,X_{n}$. Assume that they all know $\mathcal{A}$ and that $L_{\lambda}%
\in\mathcal{A}$, but that they have their own private, auxiliary information. The first statistician
knows the level $\lambda$ and therefore makes use of the set-valued excess mass estimator
\begin{equation}\label{en}
L_{1,n}\in\underset{A\in\mathcal{A}}{\arg\max}\left\{  P_{n}(A)-\lambda
\mu(A)\right\}
.
\end{equation}
The second one knows $p_{\lambda}$ and then makes use of the minimum volume estimator
\begin{equation}\label{vn}
L_{2,n}\in\underset{A\in\mathcal{A}}{\arg\min}\left\{  \mu(A):P_{n}%
(A)\geq  p_{\lambda}\right\}. 
\end{equation}
The third statistician knows $v_{\lambda}$ and thus makes use of the maximum probability estimator
\begin{equation}\label{pn}
L_{3,n}\in\underset{A\in\mathcal{A}}{\arg\max}\left\{  P_{n}(A):\mu
(A)\leq  v_{\lambda}\right\}.
\end{equation}
We  assume that $P$ and $\mathcal{A}$ are such that almost surely an 
 $L_{1,n}$ and an  $L_{2,n}$ exist and that $P_n(L_{2,n})=\lceil np_\lambda \rceil/n$. Since $P_n$ takes at most $n+1$ values, an  $L_{3,n}$  always exists.  
If $L_{j,n}$, $j=1,2,3,$ are not unique, just choose any maximizer/minimizer. It will be shown that the choice does not matter since they are indistinguishable asymptotically. 

\subsection{Overview of the results}\label{chally} What can be put forward before introducing more precisely our geometrical and probabilistic framework is as follows.\smallskip

\noindent\textbf{Convergence of random sets.} In order to compare the performance of the empirical sets  $L_{j,n}$ we study the joint limiting behavior of $L_{j,n}\bigtriangleup L_{\lambda}$, $j=1,2,3$, where $L\bigtriangleup L'=(L\cup L')\setminus(L\cap L')$ denotes the symmetric difference. The ensuing non-classical asymptotics for these set-valued estimators goes beyond the usual statistical risk approach which only provides rates for the random variables $P(L_{j,n}\bigtriangleup L_{\lambda})$ or $\mu(L_{j,n}\bigtriangleup L_{\lambda})$, for $j=1,2,3$. Instead we address the question of the weak convergence of the random sets $L_{j,n}\bigtriangleup L_{\lambda}$ themselves. We then have to design an appropriate setting allowing to state  central limit theorems for 
random sets, that is,  for sets properly centered and then magnified at a diverging scale. Our joint limit results  reveal, when magnifying with $n^{1/3}$, how the three empirical sets $L_{j,n}$ asymptotically differ or coincide. 
In particular we find that $L_{2,n}$ and $L_{3,n}$ are asymptotically indistinguishable.
Note that in the literature these limit theorems  have been considered for dimension one only, where the sets are intervals  which can be represented by two numbers,  like in the estimation of the shorth.  Hence those central limit theorems can be stated in the usual way, 
see, e.g., \cite{KP90}.\smallskip

\noindent\textbf{A local empirical process approach.} In order to analyze how the estimators $L_{j,n}$  oscillate around $L_{\lambda}$ we first show that they concentrate at rate $n^{-1/3}$  under regularity conditions that are satisfied in most of the natural settings.  Then we use an appropriate boundary empirical process, see \cite{K07}, \cite{KW08}, and \cite{EK11} and study its weak convergence on a  ``cylinder space" associated with the boundary $S_{\lambda}$ of $L_{\lambda}$. The relevant sets of $\mathcal{A}$ have to be close to $L_{\lambda}$ in Hausdorff distance at scale $n^{-1/3}$.
Interestingly, the local nature of the convergence makes both the rate dimension-free and the  Wiener process, appearing in the limit, 
distribution-free. 

\smallskip

\noindent\textbf{Organization.} Section 2 is devoted to the setup of the paper, including the relevant definitions, notation, and assumptions.   In Section 3 we present and discuss the main results and provide  a few explicit, illuminating examples. The proofs are deferred to Section 4.

\section{Setup, notation and assumptions}

\subsection{The geometrical framework and condition $H_1$}\label{21}
In order to define the appropriate limit setting the following notation and definitions are needed.

\smallskip

\noindent\textbf{The magnification map $\tau_{\varepsilon}$.} Let
$\left\Vert x\right\Vert $ denote the Euclidean norm of $x\in\mathbb{R}^{d}$ and $U=\{u: \left\Vert u\right\Vert =1\}$ the unit sphere.
Since $L_{\lambda}$ is a convex body, the metric projection $\Pi(x)\in S_{\lambda}$ of $x\in \mathbb{R}^{d}$ on
$S_{\lambda}=\partial L_{\lambda}$ is unique except for so-called skeleton points $x\in L_{\lambda}^{\ast}\subset L_{\lambda}$ with $\mu(L_{\lambda}^{\ast})=0$.
A unit vector $u \in U$ is called an outer normal of $L_\lambda$ at $\pi \in S_\lambda$ if there is some $x\in \mathbb{R}^d\setminus L_\lambda$ such that $\pi=\Pi(x)$ and $u=(x-\Pi(x))/||x-\Pi(x)||$. At each $\pi \in S_\lambda$, we denote the non-empty set of outer normals by $N(\pi)$ and write $S_{\lambda}^*=\{\pi\in S_{\lambda} :card(N(\pi))>1\}$. Note that $\mu(S_\lambda)=0$ and hence $\mu(S_\lambda^*)=0$. The normal bundle of $L_\lambda$ is
$$Nor(L_\lambda)=\{(\pi, u): \pi \in S_\lambda, u\in  N(\pi)\}.$$
As in \cite{K07} and \cite{EK11} define the magnification map
$\tau_{\varepsilon}$ at magnitude $\varepsilon>0$ to be%
\begin{equation}
\tau_{\varepsilon}(x)=\left(  \Pi(x),u(x),\frac{s(x)}{\varepsilon}\right)
\in Nor(L_\lambda) \times \mathbb{R},
\quad\text{for } x\in\mathbb{R}^{d}\setminus (L_{\lambda}^*\cup S_\lambda^*),\label{teps}%
\end{equation}
where $x=\Pi(x)+s(x)u(x)$, with $s(x)=\sgn(x-\Pi(x))||x-\Pi(x)||$ the signed distance between $x$ and $\Pi(x)$. 
\smallskip

\noindent\textbf{The cylinder space.} 
Define 
$\Sigma=Nor(L_\lambda)\times\mathbb{R}$.
Let $\nu_{d-1}$ denote both the Hausdorff surface measure on $S_{\lambda}$
(putting no mass at $S_{\lambda}^{\ast}$) and  its canonical extension to $Nor(L_\lambda)$ supported by the product Borel $\sigma$-algebra $\mathcal{G}_{d-1}$ on $S_{\lambda}\times U$. Thus, $\nu_{d-1}$ on $Nor(L_\lambda)$ is the so-called first
support measure, and we have $0<s_\lambda=\nu_{d-1}(S_{\lambda})=\nu_{d-1}(Nor(L_\lambda))<\infty$. Let $\mu_{1}$ be  Lebesgue measure on $\mathbb{R}$. The cylinder space $(\Sigma,\mathcal{F},M,d)$ is defined to be $\Sigma$ endowed with the product Borel $\sigma$-algebra $\mathcal{F}=\mathcal{G}_{d-1}\times \mathcal{B}(\mathbb{R})$, the  $\sigma$-finite product measure $M$ and the semi-metric $d$ given by
\begin{equation}
M=\nu_{d-1}\times\mu_{1}\label{m},\quad d(B,B^{\prime})=(M(B\bigtriangleup B^{\prime}))^{1/2}, \quad\text{for }B,B^{\prime}\in \mathcal{F}.
\end{equation}
For $c>0$   denote $\Sigma_{c}= Nor(L_\lambda)\times\left[  -c,c\right]$ and $\mathcal{F}_c=\left\{B\in\mathcal{F}:B\subset\Sigma_{c} \right\}$.

\smallskip

\noindent\textbf{The sufficiently parallel sets $\mathcal{A}^\varepsilon$.} 
Given $\varepsilon >0$ the $\varepsilon$-parallel set
of $S_{\lambda}$ is defined by  $S_{\lambda}^{\varepsilon}=\left\{  x:\left\Vert
x-\Pi(x)\right\Vert \leq \varepsilon\right\}  $ and we consider the sets in $\mathcal{A}$ that are ``sufficiently parallel" to $L_\lambda$,
\begin{equation}
\mathcal{A}^{\varepsilon}=\left\{ A\in\mathcal{A}:A\bigtriangleup L_{\lambda}\subset S_{\lambda}^{\varepsilon}\right\}, \quad \mathcal{C}^{\varepsilon}=\left\{ A \bigtriangleup L_{\lambda}: A\in\mathcal{A}^{\varepsilon}\right\}
.\label{Aeps}%
\end{equation}
Define the  set-to-set mapping  $$\tau_\varepsilon(C)=\{\tau_\varepsilon(x): x\in C\setminus(L_{\lambda}^*\cup S_\lambda^*)\}, \quad C\in\mathcal{B}(\mathbb{R}^d), $$
and the inverse $\tau^{-1}_\varepsilon(B)=\{x\in \mathbb{R}^d:\tau_\varepsilon(x) \in B\}$, for $B\in\mathcal{F}$. Note that 
$\tau^{-1}_\varepsilon(\tau_\varepsilon(C))=C\setminus (L_{\lambda}^*\cup S_\lambda^*).$
For $B\in \mathcal{F}$, define $\varphi_\varepsilon(B)$ to be the closure of $\tau^{-1}_\varepsilon(B)\bigtriangleup (L_\lambda \setminus (L_{\lambda}^*\cup S_\lambda^*)).$ For $A\in \mathcal{A}$, we then have $\varphi_\varepsilon(\tau_\varepsilon(A\bigtriangleup L_\lambda))=A.$

\smallskip

\noindent\textbf{The limiting class $\mathcal{B}$.} We need to magnify with $\varepsilon=n^{-1/3}$. Define for $c>0$ \begin{equation}\mathcal{B}_{c,n}=\tau_{n^{-1/3}}(\mathcal{C}^{cn^{-1/3}})=\{\tau_{n^{-1/3}}(A \bigtriangleup L_{\lambda}): A \in \mathcal{A}^{cn^{-1/3}}\}\label{bcn}
\end{equation}
and  $\mathcal{B}=\bigcup\nolimits_{c>0}\mathcal{B}_{c}$ where
\begin{equation}
\mathcal{B}_{c}=\left\{  B\in\mathcal{F}_{c}:\text{for some }B_{n}
\in\mathcal{B}_{c,n},\ \lim_{n\rightarrow \infty}\ d(B,
B_{n})=0\right\}. \label{Bc}
\end{equation}
Since $L_{\lambda}\in\mathcal{A}^{c n^{-1/3}}$ 
we have $\mathcal{B}\neq\emptyset$.
In the language of \cite{K07} each $B\in\mathcal{B}_{c}$ is a
derivative at $0$ of the set-valued function $\varepsilon\mapsto \tau
_{\varepsilon}(\mathcal{A}^{c\varepsilon})$ along the sequence $\varepsilon=n^{-1/3}$. Such limits are not uniquely determined. Actually the limit ``set" $B$ is an equivalence class of sets having $d$-distance equal to 0. Out of every equivalence class, we choose (only) one limit set $B\in\mathcal{F}_c$. This makes $d$ a metric on $\mathcal{B}_c$ and on $\mathcal{B}$. (The choices of the limit set matter. In applications we choose $B$'s such that the assumptions of our theorems are satisfied.)
Let us further assume that, for any $c>0$, $(\mathcal{B}_{c},d)$ is compact and
\begin{equation}
\label{unif}
\lim_{n\rightarrow \infty}\sup_{B_{n}\in\mathcal{B}_{c,n}}\inf_{B\in\mathcal{B}_{c}}d(B_n,B)=0.
\end{equation}

\noindent\textbf{Donsker classes.} Define 
$d_n(A,A')=
(n^{1/3}P(A\triangle A'))^{1/2}$. For $c>0$, let $\left[  \mathcal{A}\right]  _{c,n}$ and  $\left[  \mathcal{B}\right]  _{c,n}$ 
be the usual bracketing numbers w.r.t.\ $d_n$  of  $\mathcal{C}^{cn^{-1/3}} $ and $\{\tau^{-1}_{n^{-1/3}}(B):B \in \mathcal{B}_c\}$, respectively; see \cite{EK11}.
We assume either that for any $c>0$ we have
\begin{eqnarray}
\label{arn}
&&\lim_{\delta\downarrow 0}\underset{n\to \infty}{\lim\sup}
\int_{0}^{\delta} 
\sqrt{
\log\left[ \mathcal{A}\right]_{c,n}(\varepsilon)}
d\varepsilon=0, \\
&&\label{gc}
n^{1/2} \sup_{A \in\mathcal{A}}  
|P_n(A)-P(A)| 
=O_\mathbb{P}(1), \quad  n \to \infty, \\
&&
\label{brn} \lim_{\delta\downarrow0}\underset{n\rightarrow \infty}{\lim\sup}
\int\nolimits_{0}^{\delta}
\sqrt{\log\left[  \mathcal{B}\right]  _{c,n}(\varepsilon)}d\varepsilon=0,
\end{eqnarray}
or that 
\begin{equation}
\label{VC}
    \mathcal{A} \mbox{ and } \mathcal{B}\text{ are Vapnik-Chervonenkis 
    (VC) classes}.
\end{equation}
We also assume that $\mathcal{A}$ and $\mathcal{B}$ are pointwise measurable.

\smallskip 

\noindent\textbf{Nested class.} Assume that for all $r>0$, all $A \in \mathcal{A}$  there exists $A_r \in \mathcal{A}$ such that 
\begin{equation}
\label{nup}
A \subset A_r, \quad \mu(A_r)=\mu(A)+r.
\end{equation}

\smallskip
Let us denote by $H_{1}$ the conditions in Subsection \ref{21},   in particular compactness of $(\mathcal{B}_c,d)$,  (\ref{unif}), ``Donsker classes", and (\ref{nup}).

\subsection{Condition $H_2$}
We will  need conditions on the behavior of  $f$ near  $S_{\lambda}$. Let $H_2$ denote the conditions $(\ref{mass})$--$(\ref{margin2})$ below. 

Define the Hausdorff distance for the Euclidean norm on $\mathbb{R}^{d}$ as
\[
d_{H}(A,A^{\prime})=\max\left(  \sup_{x\in A}\inf_{x^{\prime}\in A^{\prime}%
}\left\Vert x-x^{\prime}\right\Vert,\sup_{x^{\prime}\in A^{\prime}}%
\inf_{x\in A}\left\Vert x-x^{\prime}\right\Vert\right)  ,\quad\text{for } A,A^{\prime}\subset\mathbb{\mathbb{R}}^{d}.
\]

\smallskip

\noindent\textbf{Excess risk of excess mass.} Consider the excess risk of excess mass
$$e_{\lambda}-e_{\lambda}(A)=P(L_\lambda)-P(A)-\lambda(\mu(L_\lambda)-\mu(A))=\int_{L_\lambda\bigtriangleup A}|f(x)-\lambda|d\mu(x). $$ We require that
for all $\delta>0$,
\begin{equation}
  \label{mass}
  \inf_{A\in \mathcal{A}  :d_H(L_\lambda, A)\geq \delta}\int_{L_\lambda\bigtriangleup A}|f(x)-\lambda|d\mu(x) >0.
\end{equation}
\smallskip

\noindent\textbf{The quadratic drift measure $D$.} We assume
that for some second-order derivatives $f_{+}^{\prime}\geq  0$ and
$f_{-}^{\prime}\geq  0$ defined on $S_{\lambda}$ we have%
\begin{align}
\lim_{\varepsilon\downarrow0}\frac{1}{\varepsilon^{2}}\int_{S_{\lambda
}^{\varepsilon}\setminus L_{\lambda}}\left\vert f(x)-\lambda+s(x)f_{+}%
^{\prime}(\Pi(x))\right\vert d\mu(x)  & =0,\label{f+}\\
\lim_{\varepsilon\downarrow0}\frac{1}{\varepsilon^{2}}\int_{S_{\lambda
}^{\varepsilon}\cap L_{\lambda}}\left\vert f(x)-\lambda+s(x)f_{-}^{\prime}%
(\Pi(x))\right\vert d\mu(x)  & =0.\label{f-}%
\end{align}
If $f$ is differentiable at $\pi\in S_{\lambda}\setminus S_{\lambda}^*$ then $f_{+}^{\prime}(\pi)=f_{-}^{\prime}(\pi)$.
Let us define on $(\Sigma,M)$ the  quadratic drift measure $D$ having density with respect to $M$ given by
\begin{equation*}
\frac{dD}{dM}(\pi,u,s)=sf_{+}^{\prime}(\pi)1_{s>0}-sf_{-}^{\prime}%
(\pi)1_{s\leq  0}\label{Df.}.
\end{equation*}

\noindent\textbf{Local excess risk of excess mass.}
Write $g(\pi)=\min(f'_+(\pi), f'_-(\pi))$ for $\pi\in S_\lambda$. Let assume that for some $\varepsilon_0>0$, $\eta_0>0$ and all $A\in\mathcal{A}$ such that $d_H(L_\lambda,A)\leq  \varepsilon_0$, we have
\begin{equation}
    \label{margin}
    \int_{L_\lambda\bigtriangleup A}|s(x)|g(\Pi(x))d\mu(x)\geq  \eta_0d_H^2(L_\lambda,A).
\end{equation}
Similarly we require that for all $B\in \mathcal{B}$  
\begin{equation}
    \label{margin2}
 D(B)\geq \eta_0  c^2(B),   
\end{equation}
where $c(B)=\inf\{c>0: B\in\mathcal{B}_c\}$. 
\smallskip

\section{Main results}

\subsection{Convergence of the excess mass set estimator}

Since $(\Sigma,M)$ is $\sigma$-finite and $(\mathcal{B},d)$ is $\sigma$-compact we can define a Wiener process $W$ indexed by $\mathcal{B}$, that is a centered Gaussian process with covariance
\[
Cov(W(B),W(B^{\prime}))=M(B\cap B^{\prime}),\quad \text{for } B,B^{\prime}%
\in\mathcal{B}.
\]
The intrinsic, standard deviation metric of $W$ on $\mathcal{B}$ is defined to be $(Var(W(B)-W(B'))^{1/2}
=d(B,B^{\prime})$. The relevant limiting random set is
\begin{equation}
Z(\mathcal{B})=\ \underset{B\in\mathcal{B}}{\arg\max
}\left\{ \sqrt{\lambda} W(B)-D(B)\right\}\label{Z}.
\end{equation}
This quantity has been studied in the univariate case where the sets reduce to numbers, see \cite{G85}, \cite{DC99}, and  \cite{BE06}.  Observe that $\mathbb{E} W^2(B) \leq 2s_\lambda c$ for $B \in \mathcal{B}_c$. 
We assume that for some $\eta _1>0$,
\begin{equation}
    \label{maxW}
    \mathbb{E}\left(\max _{B \in \mathcal{B}_c}W^2(B) \right) < \eta _1 c, \quad \text{for all } c >0.
\end{equation}

\begin{proposition}\label{Zwelldefined} 
Assume that $H_1$, $H_2$, and (\ref{maxW}) hold. With probability one, the random set $Z(\mathcal{B})$ of (\ref{Z}) exists and is unique.
\end{proposition}

We are now ready to state our non-standard weak convergence result for the sequence of random sets  $L_{1,n}$  in (\ref{en}). 

\begin{theorem}\label{thL1}
Assume that 
$H_1$, $H_2$, and (\ref{maxW}) hold. Then on some probability space there exists a triangular array $X_{n,1}, \ldots, X_{n,n}$, $n\in \mathbb{N},$ of rowwise independent
random vectors with law $P$ on $\mathbb{R}^{d}$ together with a sequence $Z_n(\mathcal{B})$ of versions of
$Z(\mathcal{B})$  such that for every argmax $L_{1,n}$ of (\ref{en}), as $n\to\infty$,
\begin{align*}
&M\left(  \tau_{n^{-1/3}}(L_{1,n}\bigtriangleup L_{\lambda
})\bigtriangleup Z_n(\mathcal{B})\right)    \stackrel{\mathbb{P}}{\to}0,\\
 &n^{1/3}\mu\left(  L_{1,n}
\bigtriangleup\varphi_{n^{-1/3}}(Z_n(\mathcal{B}))\right)  \stackrel{\mathbb{P}}{\to}0,\\
 &n^{1/3}P\left(  L_{1,n}
\bigtriangleup\varphi_{n^{-1/3}}(Z_n(\mathcal{B}))\right)   \stackrel{\mathbb{P}}{\to}0.
\end{align*}
\end{theorem}

Theorem \ref{thL1} states that, at the scale $n^{-1/3}$, the symmetric difference between the
empirical excess mass  set and $L_\lambda$ has as a limiting distribution that of the argmax of a drifted Wiener process, as defined in (\ref{Z}). 

\subsection{Convergence of  the  minimum volume set and the maximum probability set estimators}
For the second main result about  $L_{2,n}$ and $L_{3,n}$ we need some more notation and assumptions. \smallskip

\noindent\textbf{The limiting class $\mathcal{B}^*$.}
Write $B^+=B\cap(Nor(L_\lambda)\times \mathbb{R}^+)$ and $B_-=B\setminus B^+$, for $B\in\mathcal{F}$.
Now define 
$$ \mathcal{B}^*=\{B\in \mathcal{B}: M(B^+)=M(B^-)\}, \quad \mathcal{B}_c^*= \mathcal{B}^* \cap \mathcal{B}_c.$$
Note that $(\mathcal{B}_c^*,d)$ is also compact.
By replacing $\mathcal{A}$ in (\ref{Aeps})--(\ref{bcn}) with
\begin{equation}
    \label{AvAp}
    \mathcal{A}_v=\{ A \in \mathcal{A} : \mu(A)=v_\lambda \},\quad \mathcal{A}_p=\{ A \in \mathcal{A} : P(A)=p_\lambda \},
\end{equation}
respectively, we define in the same way the classes $\mathcal{A}_v^\varepsilon$, $\mathcal{C}_v^\varepsilon$, $\mathcal{A}_p^\varepsilon$, $\mathcal{C}_p^\varepsilon$ and  $\mathcal{B}^v_{c,n}$, $\mathcal{B}^p_{c,n}$.
We 
assume
\begin{align}
\label{Bv_dense}
\lim_{n\rightarrow \infty}\sup_{B\in\mathcal{B}^*_{c}}\inf_{B_n\in\mathcal{B}^v_{c,n}}d(B,B_n)&=0,\\
\label{Bp_dense}
\lim_{n\rightarrow \infty}\sup_{B\in\mathcal{B}^*_{c}}\inf_{B_n\in\mathcal{B}^p_{c,n}}d(B,B_n)&=0.
\end{align}
Consider the Wiener process $W$ indexed by $\mathcal{B}^*$ and define
\begin{equation*}
Z(\mathcal{B}^*)=\underset{B\in\mathcal{B}^{*}}{\arg\max
}\left\{ \sqrt{\lambda} W(B)-D(B)\right\}\label{Z*}.
\end{equation*}
As in Proposition \ref{Zwelldefined},
under 
$H_1$, $H_2$ and (\ref{maxW}), with probability one $Z(\mathcal{B}^*)$ exists and is unique.

In order to control the minimum volume set  estimator  we need the following two conditions.
The class $\mathcal{A}$ contains a ``univariate" subset
\begin{equation}\label{uni} \mathcal{A}_l=\{A_s\in \mathcal{A}: s\in (-p_\lambda, 1-p_\lambda), P(A_s)=p_\lambda+s\}\end{equation}
with the properties that $A_s\subset A_{s'}$ for $s<s'$, $ A_0=\L$, and for some $s_0>0, \zeta>0$ and for all 
$-s_0\leq s\leq s_0$: \begin{equation}\label{uni2} e_\lambda-e_\lambda(A_s)\leq \zeta s^2.\end{equation}

For every $c>0$, we have as $n \to\infty$,
\begin{equation}\label{clo}\sup_{A\in \mathcal{A}^{cn^{-1/3}}, P(A)=p_\lambda}\ \ \ \inf_{\tilde A\in \mathcal{A}^{cn^{-1/3}} ,P_n(\tilde A)=\lceil np_\lambda\rceil/n} d( \tau_\varepsilon(A), \tau_\varepsilon(\tilde A))  \stackrel{\mathbb{P}}{\to} 0.\end{equation}

\begin{theorem}\label{thL23}
Assume that
$H_{1}$, $H_{2}$, (\ref{maxW}), (\ref{Bv_dense}) - (\ref{clo}) hold. Then on some probability space there exists a triangular array $X_{n,1}, \ldots, X_{n,n},$ $n\in \mathbb{N},$ of rowwise independent
random vectors with law $P$ on $\mathbb{R}^{d}$ together with a sequence $Z_n(\mathcal{B}^*)$ of versions of $Z(\mathcal{B}^*)$ such that every argmin $L_{2,n}$ of (\ref{vn}) and every argmax $L_{3,n}$ of (\ref{pn}) satisfy, for $j=2,3$, as $n\to\infty$,
\begin{align*}
&M\left(  \tau_{n^{-1/3}}(L_{j,n}\bigtriangleup L_{\lambda
})\bigtriangleup Z_n(\mathcal{B}^*)\right)    \stackrel{\mathbb{P}}{\to}0,\\
 &n^{1/3}\mu\left(  L_{j,n}
\bigtriangleup\varphi_{n^{-1/3}}(Z_n(\mathcal{B}^*))\right)  \stackrel{\mathbb{P}}{\to}0,\\
 &n^{1/3}P\left(  L_{j,n}
\bigtriangleup\varphi_{n^{-1/3}}(Z_n(\mathcal{B}^*))\right)   \stackrel{\mathbb{P}}{\to}0.
\end{align*}
\end{theorem}

Comparing Theorems \ref{thL1} and \ref{thL23} we see that the limiting behavior of $L_{2,n}$ and  $L_{3,n}$  is substantially ``less rich" than that of $L_{1,n}$. The symmetry of the sets in $\mathcal{B}^*$ shows that for $j=2,3$ the inner and outer differences $L_{j,n}\setminus L_{\lambda}$ and $L_{\lambda}\setminus L_{j,n}$ tend to compensate.
Theorems \ref{thL1} and \ref{thL23} could be stated jointly since they can indeed be proved  with the same sequence of underlying Wiener processes $W_n$.
It is beyond the scope of this paper to  study $Z(\mathcal{B})$ and $Z(\mathcal{B}^*)$
in more detail, to see which argmax is ``closer" to, say,  $\Sigma_0$  (corresponding to $L_\lambda$), that is, which estimator performs better. However, a small simulation study for one-dimensional data shows that in that case $L_{2,n}$ and  $L_{3,n}$ asymptotically outperform  $L_{1,n}$.

From the proof of Theorem \ref{thL23} it follows    that  the sequence of versions $Z_n(\mathcal{B}^*)$ can be chosen the same for $L_{2,n}$ and $L_{3,n}$. Hence, we obtain, as stated in the next result, that $L_{2,n}$ and  $L_{3,n}$  are asymptotically equivalent.

\begin{corollary}
\label{cor23} Under the assumption of Theorem \ref{thL23},
as $n\to\infty$,
\begin{align*}
M(\tau_{n^{-1/3}}(L_{2,n}\bigtriangleup L_{3,n}
)
)  \stackrel{\mathbb{P}}{\to}0,\\
 n^{1/3}\mu\left(  L_{2,n}\bigtriangleup L_{3,n}\right)  \stackrel{\mathbb{P}}{\to}0,\\
 n^{1/3}P\left(  L_{2,n}\bigtriangleup L_{3,n}\right)   \stackrel{\mathbb{P}}{\to}0.
\end{align*}
\end{corollary}

\smallskip

\subsection{Discussion and examples}\label{almost}
The conditions on the class $\mathcal{A}$ are such that natural classes, like in particular the class of all closed ellipsoids, are included. If the class is ``small", e.g., by allowing not all  or only a few positive values for $\mu(A)$ or   for  $P(A)$ we can obtain  pathological and/or degenerate behavior of the set-valued estimators.
E.g., if $\mathcal{A}$ contains $L_\lambda$ and further only sets with $\mu(A)>v_\lambda$, then $L_{3,n}=L_\lambda$.

The assumptions in (\ref{f+}) and (\ref{f-}) consider the ``most regular" behavior of the density $f$ near $S_\lambda$. They lead to the cube root asymptotics in this paper. Faster or slower convergence rates are also possible, see, e.g., \cite{P95}. 
This would lead to $W$ drifted by a non-quadratic measure on the cylinder space, generalizing $W$ drifted by a convex power function used in \cite{BE06} to control the estimation of the shorth, the minimum volume convex set on the real line.
It is the goal of the present paper, however, to reveal the asymptotic theory in the most regular setup, and not to present the most general results under the weakest assumptions.


\noindent 

\smallskip
We now present some specific examples of classes of sets and probability distributions where the three level set estimators can be used.
\smallskip 

\noindent\textbf{Ellipsoids.} The natural and most studied example is the case where $\mathcal{A}$ is the class of all closed ellipsoids with non-empty interior and $P$ is an elliptical probability distribution. More, relevant details about this class of sets for the bivariate case when $L_\lambda$ is the unit disc are given in Example 1a in \cite{EK11}. In particular $\mathcal{B}_c$ is determined therein. A more restricted class is the class of all closed balls.
\smallskip


\noindent\textbf{Convex polytopes.} Another natural choice for $\mathcal{A}$ is the class of all closed, convex polytopes. In particular in  dimension 2, the class of all closed, convex quadrangles can be considered. In this case we could take a density $f$ such that $L_\lambda$ is a  rectangle. An interesting difference with the previous example is that $L_\lambda$ is non-smooth here, resulting, e.g., in a non-empty skeleton $L_\lambda^*$.
\smallskip

\noindent\textbf{Planar convex sets.} For dimension two, we can let $\mathcal{A}$ be the class of all closed, convex sets. Since this class is much larger than those in the previous examples, the restriction on $f$ that $L_\lambda$ is a convex body is much weaker now.
For this and the previous example, see again \cite{EK11}, Example 2,  for more details; in particular $\mathcal{B}_c$ is determined therein in case $L_\lambda$ is the unit square.

\smallskip

It might be difficult to determine
$L_{1,n}$, $L_{2,n}$ and $L_{3,n}$ and therefore some more flexibility in their definitions could be convenient.
Consider for instance the following ``relaxed" maximizers/minimizers:  given any  sequence 
$\delta _{n}$ of positive numbers converging to 0, choose random  sets $R_{1,n}$, $R_{2,n}$, and $R_{3,n}$ in $\mathcal{A}$  such that $P_{n}(R_{2,n})\geq p_{\lambda }$, $\mu
(R_{3,n})\leq v_{\lambda }$, and 
\begin{eqnarray*}
&&P_{n}(R_{1,n})-\lambda \mu (R_{1,n}) \geq \sup \{P_{n}(A)-\lambda \mu (A):A\in 
\mathcal{A}\}-\delta _{n}n^{-2/3}, \\
&&\mu (R_{2,n}) \leq \inf \{\mu (A):A\in \mathcal{A},P_{n}(A)\geq p_{\lambda
}\}+\delta _{n}n^{-2/3}, \\
&&P_{n}(R_{3,n}) \geq \sup \{P_{n}(A):A\in \mathcal{A},\mu (A)\leq v_{\lambda
}\}-\delta _{n}n^{-2/3}.
\end{eqnarray*}
Our approach and convergence results naturally   extend to $R_{j,n}, j=1,2,3$, but their detailed analysis is beyond the scope of this paper.  Whenever $\delta_n$ is chosen not too small (i.e., $\delta_n n^{1/3}\to \infty$) more flexible algorithms for the computation of $R_{j,n}$ could be used.

\color{black}


\section{Proofs} 
We first collect various  lemmas for the proof of the theorems.  From now on we write $\varepsilon_n=n^{-1/3}$.

\subsection{Distances, measures and drift}
\label{geometry}
For $j=1,...,d$,  let $\nu _{d-j}(\cdot )$ denote the $j$-th support measure of $L_{\lambda }$ on $Nor(L_{\lambda })$,  see \cite{S93} and \cite{SW08}. These finite measures carry  the geometrical information about $L_\lambda$. 
The local inner reach  at $\pi \in S_\lambda$ is the largest radius $r(\pi)$ of a ball  
included in $L_{\lambda }$ that has $\pi $ as a boundary point. Theorem 1\ in \cite{KW08} states a general Steiner formula for convex bodies: for any $g\in L_{1}(\mu )$,
\begin{eqnarray}
&&\int_{\mathbb{R}^{d}}g(x)d\mu (x) =\sum_{j=1}^{d}\binom{d-1}{j-1}\Theta
_{d-j}(g),  \label{SteinerGlob} \mbox{ where} \\
&&\Theta _{d-j}(g) =\int_{Nor(L_{\lambda })}\int_{-r(\pi )}^{\infty
}s^{j-1}g(\pi +su)d\mu_{1}(s)d\nu _{d-j}(\pi ,u).  \label{SupportMeasure}
\end{eqnarray}
It follows from (\ref{f+})-(\ref{f-}) and this Steiner formula with $g=f 1_{S_\lambda^\varepsilon}$, for small $\varepsilon>0$, that
\begin{equation}\label{supfinite}
    \int_{Nor(L_{\lambda })}f_{\pm }^{\prime }(\pi )\nu
_{d-j}(\pi ,u)<\infty, \
 \mbox{ for }  j=1,...,d.  \end{equation}
 Define
\begin{equation}\label{def+} \mathcal{B}_{c,n}^{p,+}=\{\tau_{\varepsilon_n}(A\bigtriangleup L_\lambda): A \in \mathcal{A}^{c\varepsilon_n},  | P(A)- p_\lambda|\leq n^{-2/5} \}.\end{equation}

\begin{lemma}\label{steiner} 
Let $c>0$. We have, as $n \to \infty$, 
\begin{equation}
\label{volume}
    \sup_{B_n \in \mathcal{B}^v_{c,n}}\left\vert M(B_n^+)-M(B_n^-)\right\vert = O(\varepsilon _n)
\end{equation}
and, if (\ref{f+}) - (\ref{f-}) hold, then
\begin{equation}
\label{symmetry}
    \sup_{B_n \in \mathcal{B}^{p,+}_{c,n}}\left\vert M(B_n^+)-M(B_n^-)\right\vert \to 0.
\end{equation}
\end{lemma}
\begin{proof} 
For $A_{n}\in \mathcal{A}^{c\varepsilon _{n}}$ and $B_{n}=\tau _{\varepsilon_{n}}(A_{n}\bigtriangleup L_{\lambda })\in \mathcal{B}_{c,n}$, we have $B_{n}^{+}=\tau _{\varepsilon _{n}}(A_{n}\setminus L_{\lambda })$ and $B_{n}^{-}=\tau _{\varepsilon _{n}}(L_{\lambda }\setminus A_{n})$. Consider $$g^+(x)=\varepsilon_{n}^{-1}1_{A_{n}\setminus L_{\lambda }}(x)=\varepsilon_{n}^{-1}1_{B_{n}^{+}}(\Pi(x), u(x),s(x)/\en )$$
in (\ref{SteinerGlob}). Then  
$$\Theta _{d-j}(g^+) =\varepsilon _{n}^{j-1}\int_{Nor(S_{\lambda})}\int_{0}^{c}s^{j-1}1_{B_{n}^{+}}(\pi, u, s)d\mu_{1}(s)d\nu _{d-j}(\pi ,u).$$
Thus by (\ref{m}), $\Theta _{d-1}(g^+)=M(B_{n}^{+})$ and $ \Theta _{d-j}(g^+)=O( \en^{j-1})$ uniformly over  $\mathcal{A}^{c\varepsilon _{n}}$, for $j=2, \dots, d$. 
Since $\varepsilon _{n}^{-1}\mu (A_{n}\setminus L_{\lambda })=\int_{\mathbb{R%
}^{d}}g^+(x)d\mu (x)$ we see that (\ref{SteinerGlob}) implies
\begin{equation}
\label{kappa1}  \sup_{A \in \mathcal{A}^{c\varepsilon _{n}}}\left\vert \frac{1}{\varepsilon _n}\mu(A\setminus L_{\lambda })-M(\tau_{\varepsilon _{n}}(A\setminus L_\lambda))\right\vert =O(\en) .
\end{equation}
Similarly we obtain
\begin{equation}
\label{kappa2}
   \sup_{A \in \mathcal{A}^{c\varepsilon _{n}}}\left\vert \frac{1}{\varepsilon _n}\mu(L_{\lambda }\setminus A)-M(\tau_{\varepsilon _{n}}(\L\setminus A))\right\vert =O(\en).
\end{equation}

For $A\in \mathcal{A}_v^{c\varepsilon _{n}}$, $\mu(A)=v_{\lambda }$ and hence $\mu(L_{\lambda }\setminus A)=\mu(A\setminus L_{\lambda })$. By (\ref{kappa1}) and (\ref{kappa2}), we obtain (\ref{volume})
by definition of $\mathcal{B}^v_{c,n}$.

Define
\begin{equation*} \mathcal{A}_{p}^{c{\en}, +}=\{ A \in \mathcal{A}^{c\varepsilon_n}: | P(A)- p_\lambda|\leq n^{-2/5} \}.\end{equation*} For $A_{n}\in \mathcal{A}_{p}^{c{\en}, +}$ we thus have  $$\left|\int\nolimits_{A_{n}\setminus L_{\lambda}}f(x)d\mu (x)-\int\nolimits_{L_{\lambda }\setminus A_{n}}f(x)d\mu (x)\right|\leq n^{-2/5},$$ and, by (\ref{f+}) and (\ref{f-}), uniformly over $\mathcal{A}_{p}^{c{\en}, +}$,
\begin{eqnarray}
&&\lambda \mu (A_{n}\setminus L_{\lambda })-\int\nolimits_{A_{n}\setminus
L_{\lambda }}s(x)f_{+}^{\prime }(\Pi (x))d\mu (x)  \label{equalprob} \\
&=&\lambda \mu (L_{\lambda }\setminus A_{n})-\int\nolimits_{L_{\lambda}\setminus A_{n}}s(x)f_{-}^{\prime }(\Pi (x))d\mu (x)+o(\varepsilon _{n}^{2})+O(n^{-2/5}).  \notag
\end{eqnarray}
Now consider $\tilde{g}^+(x)=\varepsilon_{n}^{-1}1_{A_{n}\setminus L_{\lambda }}(x)s(x)f'_+(\Pi(x))$. Then by (\ref{SteinerGlob})-(\ref{supfinite}) we find that 
uniformly over $\mathcal{A}_{p}^{c{\en}, +}$, $\varepsilon_{n}^{-1}\int\nolimits_{A_{n}\setminus
L_{\lambda }}s(x)f_{+}^{\prime }(\Pi (x))d\mu (x)=O(\en)$. We can deal similarly with the integral on $\L\setminus A_n$. Using this in  (\ref{equalprob}) in combination with  (\ref{kappa1}) and (\ref{kappa2}) yields (\ref{symmetry}). \end{proof}

Observe that  (\ref{kappa1}) and (\ref{kappa2})  immediately yield
\begin{equation}\label{mismu}
\sup_{A \in \mathcal{A}^{c\varepsilon _{n}}}\left\vert \varepsilon _n^{-1}\mu(A\bigtriangleup L_\lambda)-M((\tau_{\varepsilon _{n}}(A\bigtriangleup L_\lambda)))\right\vert=O(\varepsilon _n).\end{equation}

\begin{lemma} \label{compact}
If $H_1$ and $H_2$ hold, then
\begin{eqnarray}
\label{Bv_compact}
&&\lim_{n\rightarrow \infty}\sup_{B_n\in\mathcal{B}^v_{c,n}}\inf_{B\in\mathcal{B}^*_{c}}d(B_n,B)=0,\\
\label{Bp_compact}
&&\lim_{n\rightarrow \infty}\sup_{B_n\in\mathcal{B}^p_{c,n}}\inf_{B\in\mathcal{B}^*_{c}}d(B_n,B)=0,\\
\label{Bp+_compact}
&&\lim_{n\rightarrow \infty}\sup_{B_n\in\mathcal{B}^{p,+}_{c,n}}\inf_{B\in\mathcal{B}^*_{c}}d(B_n,B)=0.
\end{eqnarray}
\end{lemma}
\begin{proof}
    If (\ref{Bv_compact}) is false, then for some  $\delta >0$ and some subsequence $n_k$ we can find sets $\tilde{B}_{n_k}\in\mathcal{B}^v_{c,n_k}$ such that $\inf_{\tilde B\in\mathcal{B}^*_{c}}d(\tilde{B}_{n_k},\tilde B)>\delta$. But because of (\ref{unif})
    and the compactness of  $(\mathcal{B}_c,d)$ one can extract a further subsequence $n_{k_j}$ and sets  $B_j$ converging w.r.t. $d$ to some $B\in\mathcal{B}_c$. 
    Lemma \ref{steiner} yields  that $\vert M(B_j^+)-M(B_j^-)\vert \rightarrow 0$ 
    which implies $B\in\mathcal{B}^*_c$ and hence the contradictory fact that $d(B_j,B)\rightarrow 0$. 
    
    The proof of  (\ref{Bp+_compact}) follows similarly. Clearly  (\ref{Bp+_compact}) implies (\ref{Bp_compact}).
\end{proof}

For $c>0$ consider $C\in \mathcal{C}^{c\en}$. Write $C^+=C\setminus \L$ and $C^-=C\cap \L$.
Define $$D_n(\tau_{\en}(C))=n^{2/3}\left(e_{\lambda}(C^-)-e_{\lambda}(C^+)\right)$$ and observe that   $e_{\lambda}(C^-)\geq 0$ and 
$e_{\lambda}(C^+) \leq 0$. 

\begin{lemma}\label{rnSteiner}
If  (\ref{f+}) - (\ref{f-}) hold, then, as $n\to \infty$, 
    $$ 
   \sup_{C\in \mathcal{C}^{c\varepsilon_n} }|D(\tau_{\en}(C))-D_n(\tau_{\en}(C))|\to 0. 
    $$
\end{lemma}



\begin{proof}  Write $f^{\prime }(\pi,s )=1_{s>0}f_{+}^{\prime }(\pi )+1_{s\leq
0}f_{-}^{\prime }(\pi )$.  From the Steiner formula (\ref{SteinerGlob})-(\ref{SupportMeasure}) and from 
(\ref{f+}) - (\ref{f-}) we obtain by a straightforward calculation that, uniformly  for $C\in \mathcal{C}^{c\varepsilon_n}$,
\begin{eqnarray*}
&&\!\!\!\!\!\!\!\!D_n(\tau_{\en}(C))\\&&\!\!\!\!\!\!\!\!\!\!\!\!\!\!\!\!=\frac{1}{\en^2}
\int_{Nor(L_{\lambda })}\int_{-(r(\pi )\wedge c\en)}^{c\en
}sf^{\prime }(\pi,s )(1_{C^+}(\pi+su)-1_{C^-}(\pi+su))dsd\nu _{d-1}(\pi ,u)\\&&
\!\!\!\!\!\!\!\!\!\!\!\!\!\!\!\!\!+\sum_{j=2}^d \binom{d-1}{j-1}\frac{1}{\en^2}
\int_{Nor(L_{\lambda })}\int_{-(r(\pi )\wedge c\en)}^{c\en
}s^jf^{\prime }(\pi,s )(1_{C^+}(\pi+su)-1_{C^-}(\pi+su))dsd\nu _{d-j}(\pi ,u)\\
&&\!\!\!\!\!\!\!\!+o(1)\\
&&\!\!\!\!\!\!\!\!=:T_{1,n}(C)+\sum_{j=2}^d \binom{d-1}{j-1}T_{j,n}(C)+o(1).
\end{eqnarray*}
Now by a change of variables it follows that $T_{1,n}(C)= D(\tau_{\en}(C))$. Hence it remains to show that 
$\sup_{C\in \mathcal{C}^{c\varepsilon_n} }\sum_{j=2}^d \binom{d-1}{j-1}|T_{j,n}(C)|\to 0$,
but  this follows  from  $\sup_{C\in \mathcal{C}^{c\varepsilon_n} } |T_{j,n}(C)|=O(\varepsilon_n^{j-1})$, which we obtain from (\ref{supfinite}).
\end{proof}

The following lemma is immediate from basic measure theory, more precisely the fact that an $M$-small set has a small integral.  
\begin{lemma}\label{driftincrement}
Assuming  (\ref{f+})-(\ref{f-}) we have, as $n \to \infty$,
    $$ 
    \sup_{B_n\in \mathcal{B}_{c,n}, B\in \mathcal{B}_{c},\, d(B_n,B)\leq  \gamma_{c,n}}|D(B_n)-D(B)|\to 0.
    $$
\end{lemma}

\subsection{Concentration lemmas}
\label{concentrations}

\begin{lemma}\label{subs}
Let $\varepsilon>0$ fixed  and  $A\in \mathcal{A}$ with $d_H(A, L_\lambda)\leq  \varepsilon$, then $A\bigtriangleup L_\lambda\subset S_\lambda^\varepsilon$.
\end{lemma}
\begin{proof}
Assume $d_H(A, L_\lambda)\leq  \varepsilon$ and $x\in A\setminus L_\lambda$.
Then $\left\Vert x-\Pi(x)\right\Vert \leq  \varepsilon$. Hence $x\in S_\lambda^\varepsilon$. Now assume $d_H(A, L_\lambda)\leq  \varepsilon$ and $x\in  L_\lambda \setminus A$. Assume $x\notin S_\lambda^\varepsilon$. Let $\Pi_A(x)$ be the orthogonal projection of $x$ on $\partial A$, that is unique since $x \notin A$ and $A$ is convex. There exists an $y\in S_\lambda$ such that $\Pi_A(y)=\Pi_A(x)$. To see this consider the tangent space of $A$ at $\Pi_A(x)$ that is orthogonal to the outer normal of $\partial A$ at $\Pi_A(x)$ driven by $(x-\Pi_A(x))$ and take $y$ as the intersection of that line with $S_\lambda$. Then $\left\Vert y-\Pi_A(x) \right\Vert >\left\Vert y-x \right\Vert\geq  \left\Vert \Pi(x)-x \right\Vert>\varepsilon$. This implies $d_H(\{y\},A)> \varepsilon$ and hence $d_H(A, L_\lambda)> \varepsilon$. Contradiction. Hence we have $x \in S_\lambda^\varepsilon$.
\end{proof}

Consider the following variant of  $L_{3,n}$:
\begin{equation}
    \label{L4}
    L_{4,n}  \in\underset{A\in\mathcal{A}}{\arg\max}\left\{  P_{n}(A): \mu(A)=v_\lambda \right\}.
\end{equation}

\begin{lemma} \label{concentration}
    Under the assumptions of Theorems \ref{thL1} or \ref{thL23}, respectively, for every $\delta>0$, there exists a $c>1$, and an $n_0$, such that, for $j=1$ and for $j=2,4$, and $n\geq  n_0$,
    $$\mathbb{P}(d_H(L_{j,n}, L_\lambda)\geq  c\varepsilon_n)\leq  \delta. $$
\end{lemma}
\begin{proof} 
    Consider  
\begin{eqnarray*}&&  L_{1,n}\in \underset{ A\in \mathcal{A}}{\arg\max} \left\{P_n(A)-\lambda\mu(A)\right\}\\
&&=\underset{ A\in \mathcal{A}}{\arg\max} 
\left\{ P_n(A)-P(A) -P_n(L_\lambda)  +P(L_\lambda)+e_\lambda(A)-e_\lambda\right\}.
\end{eqnarray*}
Observe that the expression of which the latter argmax is taken is equal to 0 in case $A=L_\lambda$.

First assume that $\mathcal{A}$ is a VC class.  We begin with showing  that for $n$ large enough 
$$\mathbb{P}(d_H( L_{1,n},\L)\geq\delta)\leq\frac{1}{2}\delta.$$
We obtain from (\ref{mass}) that there  exists an $\eta>0$, such that 
$d_H(\L,A))\geq \delta$ implies $e_\lambda-e_\lambda(A) \geq 2\eta$. 
The Glivenko-Cantelli theorem on $\mathcal{A} $ yields that for  the above $\eta$  for large $n$
$$\mathbb{P}(\sup_{A\in \mathcal{A}} |P_n(A)-P(A) -P_n(L_\lambda)  +P(L_\lambda)|\leq \eta)\geq 1-\frac{1}{2}\delta.$$
Hence $$\mathbb{P}(d_H( L_{1,n},\L)<\delta)\geq 1-\frac{1}{2}\delta.$$

Define, for $c>1$, $$p_c= \mathbb{P}
(c\en \leq 
d_H( L_{1,n},\L)
\leq \min(c^2\en,\delta) ).
$$
and $ \mathcal{A}_{1}=\{A\in \mathcal{A}: c\en \leq 
d_H( \L, A)
\leq \min(c^2\en,\delta)\}$.
From (\ref{margin}) we obtain for large $n$,
$$e_\lambda-e_\lambda(A)>\frac{1}{2}\eta_0 d_H^2(\L,A), \quad\mbox {for small } d_H(\L,A).$$
Hence for large $n$
$$\inf_{A\in \mathcal{A}_{1} }
e_\lambda-e_\lambda(A)\geq \frac{1}{2}\eta_0 \inf_{A\in \mathcal{A}_{1} }d_H^2(\L,A) \geq\frac{1}{2}\eta_0 c^2\en^2.$$
This yields  
\begin{eqnarray*}p_c&\le& \mathbb{P}\left(\sup _{A\in \mathcal{A}_{1} }P_n(A)-P(A) -P_n(L_\lambda)  +P(L_\lambda)\geq \inf_{A\in \mathcal{A}_{1} } e_\lambda-e_\lambda(A)\right) \\
&\le& \mathbb{P}\left(\sup _{A\in \mathcal{A}_{1} }P_n(A)-P(A) -P_n(L_\lambda)  +P(L_\lambda)\geq \frac{1}{2}\eta_0 c^2\en^2\right)
\\&\le& \mathbb{P}\left(2\sup _{D\in \mathcal{D}_{n} }|P_n(D)-P(D)| 
\geq \frac{1}{2}\eta_0 c^2\en^2\right),
\end{eqnarray*}
where $$\mathcal{D}_{n}=\{A\setminus \L:A\in \mathcal{A}_1 \}\cup
\{\L\setminus A:A\in \mathcal{A}_1 \}.$$

Now, very similar as in    the proof of Theorem 2 in \cite{EK11}, we obtain, using Lemma \ref{subs}, that the latter probability is bounded by $$c_1\exp(-c_2\eta_0^2c^2),$$ for some constants $c_1,c_2>0.$

Using this bound on $p_c$ with $c$ replaced by $c^{2^m}$, $m=0,1, 2, \ldots$, we obtain that for large $n$ 
\begin{eqnarray*}&&\mathbb{P}(d_H(L_{1,n}, L_\lambda)\geq  c\varepsilon_n) \\
&&\le\mathbb{P}(d_H(L_{1,n}, L_\lambda)\geq  \delta)
+\sum_{m=0}^\infty \mathbb{P}
(c^{2^m}\en \leq 
d_H( L_{1,n},\L)
\leq \min(c^{2^{m+1}}\en,\delta) )\\
&&\leq \frac{1}{2} \delta +c_1\sum_{m=0}^\infty \exp(-c_2\eta_0^2c^{2^{m+1}})\leq \delta, \end{eqnarray*}
if $c$ is large enough.

In case (\ref{arn}) and (\ref{gc}) hold, the proof for $L_{1,n}$ follows the same lines, but now the arguments in the proof of Theorem 1 in \cite{EK11} should be used, in particular the application of Lemma 19.34 in \cite{vdV98}.

Next we consider $L_{4,n}$. We have 
$$ L_{4,n}\in \underset{A\in\mathcal{A},\mu(A)=v_\lambda}{\arg\max}\left\{  P_n(A)-P(A) -P_n(L_\lambda)  +P(L_\lambda)+e_\lambda(A)-e_\lambda\right\}.$$
This expression is very similar to the one for $L_{1,n}$. The only difference is that    $\mathcal{A}$ there is replaced by 
its  subset $\{A \in \mathcal{A}:  \mu(A)= v_\lambda \}$. Since the arguments above  - dealing with suprema and infima - hold for the entire class $\mathcal{A}$, they remain to hold for this subset.

Finally  consider $L_{2,n}$. We have, almost surely,
\begin{align*}
& L_{2,n}\in \underset{A\in\mathcal{A}, nP_n(A)=\lceil np_\lambda \rceil}{\arg\min}\left\{ \mu(A)\right\}  \\
& =\underset{A\in\mathcal{A}, nP_n(A)=\lceil np_\lambda \rceil}{\arg\max}\left\{  P_{n}(A)-\lambda\mu(A)\right\}  \\
& =\underset{A\in\mathcal{A},nP_n(A)=\lceil np_\lambda \rceil}{\arg\max}\left\{  P_n(A)-P(A) -P_n(L_\lambda)  +P(L_\lambda)+e_\lambda(A)-e_\lambda\right\}.
\end{align*}
This expression looks  similar to the ones for $L_{1,n}$ and $L_{4,n}$, but the difference is that the supremum of the expression of which the latter argmax is taken is not guaranteed to be non-negative since the choice  $A=\L$, as before, is not allowed. However, it follows from (\ref{uni})  that, almost surely, there exists an $A_{\hat s}\in\mathcal{A}_l$ such that $nP_n(A_{\hat s})=\lceil np_\lambda \rceil$. 
Then, using $P(A_{\hat s})=p_\lambda+O_\mathbb{P}(1/\sqrt{n})$, we obtain  
from   (\ref{uni2}) and the behavior of the oscillation modulus of the univariate, uniform empirical process,  
that with arbitrarily high probability for large $n$ that the just mentioned supremum is larger than $-n^{-17/24}$ (instead of being non-negative). Since $n^{-17/24}/\en^2\to 0$ as $n\to \infty$, the proof for $L_{1,n} $ can be easily adapted, replacing $\mathcal{A}$ by its (random) subset 
$\{A \in \mathcal{A}:  nP_n(A)=\lceil np_\lambda \rceil \}$.
 \end{proof}

\subsection{Processes on the cylinder space}
\label{processes}
Here we describe more precisely the local objects, magnified into the cylinder space, namely the empirical process, the drift induced by the local variation of the density, and then the limiting drifted Gaussian process. 

Since for all $c>0$, 
$(\mathcal{B}_{c},d)$ is totally bounded, we have
$$\sup_{B\in\mathcal{B}_{c}}\inf_{B_{n}\in\mathcal{B}_{c,n}%
}d(B_{n}, B)\to 0.
$$
Combining this with
 (\ref{unif}), we have in terms of
Hausdorff distance between classes of sets that for any 
$c>0$, as  $n\rightarrow\infty$,
$$
\gamma_{c,n}:=
\max\left(  \sup_{B_{n}%
\in\mathcal{B}_{c,n}}\inf_{B\in\mathcal{B}_{c}}d(B_{n},B),\sup
_{B\in\mathcal{B}_{c}}\inf_{B_{n}\in\mathcal{B}_{c,n}}d(B_{n},B)\right) \\
\to 0.
$$
Define 
$$\Lambda_n(C)=n^{2/3}(P_n(C)-P(C)),\quad C\in \mathcal{B}(\mathbb{R}^d),$$
and 
$$w_n(B)=\Lambda_n(\tau_{\varepsilon_n}^{-1}(B^+))-\Lambda_n(\tau_{\varepsilon_n}^{-1}(B^-)),\quad B\in \mathcal{F}_c\, .$$

\begin{lemma}\label{embedding}
Assume that $H_{1}$, and $H_{2}$ hold. Let $c>0$. Then on some probability space there exists a triangular array $X_{n,1}, \ldots, X_{n,n}, n\in \mathbb{N},$ of rowwise independent
random vectors with law $P$ on $\mathbb{R}^{d}$ together with a bounded, $d$-continuous version of
$W$ on $\mathcal{B}_c$ such that,  as $n\to\infty$,
\begin{equation} \label{wnproba}
    \sup_{B_n\in \mathcal{B}_{c,n}, B\in \mathcal{B}_{c}, d(B_n,B)\leq  \gamma_{c,n}}|w_n(B_n)-w_n(B)|\stackrel{\mathbb{P}}{\to}0,
\end{equation}
and, with probability 1,
\begin{equation} \label{wnapprox}
\sup_{B\in \mathcal{B}_{c} }|w_n(B)-\sqrt{\lambda}W(B)|\to 0.
\end{equation}
\end{lemma}

\begin{proof}
Note that the assumptions of Theorems 1 and 2  in \cite{EK11} are satisfied. Hence, using these theorems, including a Skorohod construction as on page 554 therein,   yields
 (\ref{wnproba}) and (\ref{wnapprox}). Note that the generalization from $c=1$ therein to arbitrary $c>0$ here, is straightforward. Also the fact that here $w_n$ is a difference of two terms can be easily dealt with.
\end{proof}



For a compact subset $\check{\mathcal{B}}$ of $\mathcal{B}$, define  $$Z(\check{\mathcal{B}})=\underset{B\in\check{\mathcal{B}}}{\arg\max
}\left\{ \sqrt{\lambda} W(B)-D(B)\right\}.$$
Recall that 
$ \sqrt{\lambda} W-D$ is $d$-continuous on $\mathcal{B}_c$, whereas $Var(W(B)-W(B'))=0$ implies $B'=B$ by our equivalence class convention.  Now note that both 
$Z(\mathcal{B}_c)$ and $Z(\mathcal{B}^*_c)$ 
exist and, by Lemma 2.6 in \cite{KP90},  are 
almost surely unique 
on the compact set $\mathcal{B}_c$, respectively $\mathcal{B}^*_c$ . 
Proposition~\ref{Zwelldefined} and a similar statement for $Z(\mathcal{B}^*)$ are  consequences  of (the above and) the following lemma. 



\begin{lemma} \label{Wcfinite}
    Assume that $H_{1}$,  $H_2$, and (\ref{maxW}) hold. For  $\tilde{\mathcal{B}}=\mathcal{B}, \mathcal{B}^*$ we have  
    $$\mathbb{P} \left( \bigcup_{c>0}\ \{Z(\tilde{\mathcal{B}}\cap \mathcal{B}_c)=Z(\tilde{\mathcal{B}}\cap\mathcal{B}_{\tilde c}), \mbox{ for all } \tilde c>c\}\right)=1.
    $$
    Hence $Z(\tilde{\mathcal{B}})$ almost surely exists and is unique; it is  the ``set limit" of\\ $Z(\tilde{\mathcal{B}}\cap \mathcal{B}_m)$:
$$Z(\tilde{\mathcal{B}})=\bigcap_{k= 1}^\infty\bigcup_{m= k}^\infty
Z(\tilde{\mathcal{B}}\cap \mathcal{B}_m).$$
\end{lemma}
\begin{proof}
  We have, using (\ref{margin2}), 
  \begin{eqnarray*}
  && \mathbb{P}\left(\sup_{B\in \tilde{\mathcal{B}}: c(B)\geq c} \sqrt{\lambda} W(B)- D(B)\geq 0\right)\\
  &&\leq \sum_{m=0}^\infty
  \mathbb{P}\left(\sup_{B\in \tilde{\mathcal{B}}: c^{2^m} \leq c(B) < c^{2^{m+1}}} \sqrt{\lambda} W(B)- D(B)\geq 0\right)\\
  &&\leq \sum_{m=0}^\infty
  \mathbb{P}\left(\sup_{B\in \tilde{\mathcal{B}}: c^{2^m} \leq c(B) < c^{2^{m+1}}} \sqrt{\lambda} W(B)- \eta_0c^2(B)\geq 0\right)\\
  &&\leq \sum_{m=0}^\infty
  \mathbb{P}\left(\sup_{B\in \tilde{\mathcal{B}}: c^{2^m} \leq c(B) < c^{2^{m+1}}} \sqrt{\lambda} W(B)\geq \eta_0c^{2^{m+1}}\right),
  \end{eqnarray*}
  which is by (\ref{maxW}), bounded from above by
$$ \frac{\sqrt{\eta_1 \lambda}}{\eta_0}
\sum_{m=0}^\infty c^{-2^m},$$ which is, for arbitrary $\eta>0$, bounded by $\eta$, for $c$ large enough.

Hence, since $L_\lambda\in \mathcal{A}$,  for $c$ large enough,  $$\mathbb{P}(Z(\tilde{\mathcal{B}}\cap \mathcal{B}_c)=Z(\tilde{\mathcal{B}}\cap\mathcal{B}_{\tilde c}), \mbox{ for all } \tilde c>c)\geq 1-\eta.$$ If this event is denoted by $\Omega_c$, then $\mathbb{P}(\cup_{c>0} \Omega_c)=1$.
\end{proof}

\color{black}


\subsection{Proof of Theorem \ref{thL1}}
\label{proofth1} We work in the setting of Lemma \ref{embedding}. For $c>0$ we have 
\begin{align*}
& \underset{A\in\mathcal{A}^{c\varepsilon
_{n}}}{\arg\max}\left\{  P_{n}(A)-\lambda\mu(A)\right\}  \\
& =\underset{A\in\mathcal{A} ^{c\varepsilon
_{n}}}{\arg\max}\left\{  P(A)-\lambda\mu(A)-P(L_{\lambda})+\lambda
\mu(L_{\lambda})+n^{-2/3}(\Lambda_{n}(A)-\Lambda_{n}(L_{\lambda}))\right\}
\\
& =\underset{A\in\mathcal{A}^{c\varepsilon
_{n}}}{\arg\max}\left\{  n^{2/3}(e_{\lambda}(A)-e_{\lambda})+\Lambda
_{n}(A)-\Lambda_{n}(L_{\lambda})\right\}  \\
& =\varphi_{\varepsilon_{n}}\left\{  \underset{B\in\mathcal{B}_{c,n}}%
{\arg\max}\left\{  w_{n}(B)-D_n(B)\right\}  \right\} 
.
\end{align*}

Consider the events
\[
\Xi_{c,n}^{\Lambda}=\left\{ L_{1,n}\bigtriangleup L_{\lambda}\subset S_\lambda^{c\varepsilon_n} 
\right\}  ,\quad\Xi_{c}^{W}=\left\{  Z(\mathcal{B}_{c})=Z(\mathcal{B}%
)\right\}  ,\]
where  $Z(\mathcal{B}_{c})$ and $Z(\mathcal{B})$ are defined in terms of a Wiener process $W$ satisfying (\ref{wnapprox}) in  Lemma~\ref{embedding}.
Clearly, Lemmas \ref{subs}, \ref{concentration} and \ref{Wcfinite} imply that for any
$\delta>0$ there exists a $c=c(\delta)>0$ such that we have $\mathbb{P}\left(
\Xi_{c,n}^{\Lambda}\cap\Xi_{c}^{W}\right)  >1-\delta$ for all $n$ large
enough. Now define%
\[
m_{c,n}=\sup_{B\in\mathcal{B}_{c,n}}\left\{  w_{n}(B)-D_n(B)\right\}  ,\quad
m_{c}=\max_{B\in\mathcal{B}_{c}}\left\{  \sqrt{\lambda}W(B)-D(B)\right\}
\]
and observe that Lemmas~\ref{embedding},  \ref{rnSteiner} and  \ref{driftincrement} imply
\begin{equation}\label{mm} m_{c,n}\stackrel{\mathbb{P}}{\to}  m_{c}, \mbox{ as } n\rightarrow\infty.
\end{equation} 

We have for any $\varepsilon>0$ fixed, every argmax $L_{1,n}$, and
all large enough $n$ 
\begin{align*}
& \mathbb{P}\left(  M(\tau_{\varepsilon_{n}}(L_{1,n}\bigtriangleup L_{\lambda})\bigtriangleup
Z(\mathcal{B}))>\delta^{2}\right)  \\
& \leq \mathbb{P}\left(  \left\{  M(\tau_{\varepsilon_{n}}(L_{1,n}\bigtriangleup
L_{\lambda})\bigtriangleup Z(\mathcal{B}))>\delta^{2}\right\}  \cap\Xi_{c,n}^{\Lambda
}\cap\Xi_{c}^{W}\right)  +\delta\\
& \leq \mathbb{P}\left(  d\left(  \underset{B\in\mathcal{B}_{c,n}}%
{\arg\max}\left\{  w_{n}(B)-D_n(B)\right\}  ,Z(\mathcal{B}_{c})\right)
>\delta\right)  +\delta\\
& \leq  \mathbb{P}\left(  \sup_{B\in\mathcal{B}_{c,n}:d(B,Z(\mathcal{B}%
_{c}))>\delta}\left\{  w_{n}(B)-D_n(B)\right\}  \geq 
m_{c,n}\right)  +\delta\\
& \leq \mathbb{P}\left(  \sup_{B\in\mathcal{B}_{c,n}:d(B,Z(\mathcal{B}%
_{c}))>\delta}\left\{  w_{n}(B)-D(B)\right\}  \geq 
m_{c,n}-\varepsilon\right)  +\delta
\end{align*}
which is by (\ref{mm})
\[
\leq \mathbb{P}\left(  \sup_{B\in\mathcal{B}_{c,n}:d(B,Z(\mathcal{B}%
_{c}))>\delta}\left\{  w_{n}(B)-D(B)\right\}  \geq  m_{c}-2\varepsilon
\right)  +2\delta%
\]
which by (\ref{wnproba}) and Lemma \ref{driftincrement} is in turn
\[
\leq \mathbb{P}\left(  \sup_{B\in\mathcal{B}_{c}:d(B,Z(\mathcal{B}%
_{c}))\geq  \delta/2}\left\{  w_{n}(B)-D(B)\right\}  \geq  m_{c}-3\varepsilon
\right)  +3\delta%
\]
and this is by (\ref{wnapprox}) and then by Lemma 2.6 in \cite{KP90}
\[
\leq \mathbb{P}\left(  \max_{B\in\mathcal{B}_{c}:d(B,Z(\mathcal{B}%
_{c}))\geq  \delta/2}\left\{ \sqrt{\lambda} W(B)-D(B)\right\}  \geq  m_{c}-4\varepsilon
\right)  +4\delta\leq 5\delta,
\]
 provided that we  choose a small enough $\varepsilon$ with respect to $\delta$.

Note that $Z(\mathcal{B})$ depends on $\delta$ through $c=c(\delta)$. We can avoid this, but make it instead depend on $n$ as in the statement of the theorem, by a  diagonal selection argument.
 
 The second and third statement in Theorem \ref{thL1} follow directly from the just established first one and 
  the  Steiner formula (\ref{SteinerGlob})-(\ref{SupportMeasure}), since $M$ can be approximated by $\varepsilon_n^{-1}\mu$ after transforming back by $\tau^{-1}_{\varepsilon_n}$ (see (\ref{mismu})), and then  $\lambda\mu$ can be approximated  by $P$ near $S_\lambda$.  $\hfill\Box$%

\subsection{Proof of Theorem~\ref{thL23}}
\label{proofth2}
 The proof of Theorem \ref{thL23} with $L_{3,n}$ replaced by $L_{4,n}$ from (\ref{L4}) is similar to that  of Theorem \ref{thL1}, only $\mathcal{A}$ has to be replaced by $\mathcal{A}_{v}$ and $\mathcal{B}$ by $\mathcal{B}^*$. 

Now take an argmax $L_{3,n}$ with $\mu(L_{3,n})<v_\lambda$. Then, using (\ref{nup}),  for some $L_{4,n}$ we have $L_{3,n}\subset L_{4,n}$. Now, since   $n^{1/2}(P_n-P)=O_\mathbb{P}(1)$ uniformly on $\mathcal{A}$, we have with probability tending to 1,
$$P(L_{3,n})\geq P_n(L_{3,n})-\frac{1}{2} n^{-2/5}\geq P_n(L_\lambda)-\frac{1}{2} n^{-2/5} \geq p_\lambda-n^{-2/5} .$$
Since $ P(L_{3,n})-\lambda \mu(L_{3,n})\leq p_\lambda-\lambda v_\lambda$
we get
$\mu(L_{3,n})\geq  v_\lambda-\frac{1}{\lambda}n^{-2/5}$
thus
$$\mu(L_{3,n}\triangle L_{4,n})= \mu( L_{4,n})-\mu(L_{3,n})\leq \frac{1}{\lambda}n^{-2/5}=o(\varepsilon_n).$$
This also implies that $M(\tau_{\varepsilon_n}(L_{3,n}\triangle L_{4,n}))\leq \en (\mu( L_{4,n})-\mu(L_{3,n}))\stackrel{\mathbb{P}}{\to} 0$ and the statements of Theorem \ref{thL23} for $j=3$ follow from those for $j=4$.

Finally we  consider $L_{2,n}$. We follow again the line of reasoning  and the notation in the proof of Theorem~\ref{thL1}. Define $$\hat{\mathcal{B}}_{c,n}=\{\tau_{\varepsilon_n}(A \bigtriangleup L_\lambda): A \in \mathcal{A}^{c\varepsilon_n}, P_n(A)=\lceil np_\lambda \rceil/n\}.$$

We have for $c>0$ 
\begin{align*}
& \underset{A\in\mathcal{A}_{c\varepsilon
_{n}}, nP_n(A)=\lceil np_\lambda \rceil}{\arg\min}\left\{ \mu(A)\right\}  \\
& =\underset{A\in\mathcal{A}_{c\varepsilon
_{n}}, nP_n(A)=\lceil np_\lambda \rceil}{\arg\max}\left\{  P_{n}(A)-\lambda\mu(A)\right\}  \\
& =\underset{A\in\mathcal{A}_{c\varepsilon
_{n}},nP_n(A)=\lceil np_\lambda \rceil}{\arg\max}\left\{  P(A)-\lambda\mu(A)-P(L_{\lambda})+\lambda
\mu(L_{\lambda})+n^{-2/3}(\Lambda_{n}(A)-\Lambda_{n}(L_{\lambda}))\right\}\\
& =\underset{A\in\mathcal{A}_{c\varepsilon
_{n}}, nP_n(A)=\lceil np_\lambda \rceil}{\arg\max}\left\{  n^{2/3}(e_{\lambda}(A)-e_{\lambda})+\Lambda
_{n}(A)-\Lambda_{n}(L_{\lambda})\right\}  \\
& =\varphi_{\varepsilon_{n}}\left(  \underset{B\in\hat{\mathcal{B}}_{c,n}}%
{\arg\max}\left\{  w_{n}(B)-D_n(B)\right\}  \right)  
.
\end{align*}

Consider the events
\[
\Xi_{c,n}^{\Lambda,*}=\left\{  L_{2,n}\triangle L_{\lambda}\subset S_\lambda^{c\varepsilon_n}\right\}  ,\quad\Xi_{c}^{W,*}=\left\{  Z(\mathcal{B}^*_c)=Z(\mathcal{B}^*%
)\right\}  .
\]
Again,  Lemmas \ref{concentration} and \ref{Wcfinite} imply that for any
$\delta>0$ there exists a $c=c(\delta)>0$ such that we have $\mathbb{P}\left(
\Xi_{c,n}^{\Lambda,*}\cap\Xi_{c}^{W,*}\right)  >1-\delta$ for all $n$ large
enough. Now define
\begin{eqnarray*}
&&\hat m_{c,n}=\sup_{B\in \hat{\mathcal{B}}_{c,n}}\left\{  w_{n}(B)-D_n(B)\right\} , \quad
m_{c,n}^p=\sup_{B\in{\mathcal{B}}_{c,n}^p}\left\{  w_{n}(B)-D(B)\right\}  ,\\
&&m_{c}^*=\max_{B\in\mathcal{B}^*_{c}}\left\{  \sqrt{\lambda}W(B)-D(B)\right\}
\end{eqnarray*}
and note that by  (\ref{clo}), the asymptotic equicontinuity of $w_n$   (as in the proof of Lemma~\ref{embedding} given in \cite{EK11}), and Lemma \ref{rnSteiner}, for $\varepsilon>0$, 
\begin{equation}\label{mcmh}\mathbb{P}( m_{c,n}^p\leq \hat m_{c,n}+\varepsilon)\to 0, \quad \mbox {as } n\to\infty,\end{equation} 
and that  by  (\ref{Bp_dense}), (\ref{Bp_compact}) and Lemma~\ref{embedding} (possibly with a larger $\gamma_{c,n}\to 0$),  
\begin{equation}\label{mmm} 
  m_{c,n}^p\stackrel{\mathbb{P}}{\to}  m_{c}^*, \mbox{ as } n\rightarrow\infty.
\end{equation} 

Recall the definition of $\mathcal{B}_{c,n}^{p,+}$ in  (\ref{def+}). We have for  $\varepsilon>0$, every argmin $L_{2,n}$, and
all large enough $n$ 
\begin{align*}
& \mathbb{P}\left(  M(\tau_{\varepsilon_{n}}(L_{2,n}\bigtriangleup L_{\lambda})\bigtriangleup
Z(\mathcal{B}^*))>\delta^{2}\right)  \\
& \leq \mathbb{P}\left(  \left\{  M(\tau_{\varepsilon_{n}}(L_{2,n}\bigtriangleup
L_{\lambda})\bigtriangleup Z(\mathcal{B}^*))>\delta^{2}\right\}  \cap\Xi_{c,n}^{\Lambda,*
}\cap\Xi_{c}^{W,*}\right)  +\delta\\
& \leq \mathbb{P}\left(  d\left(  \underset{B\in \hat{\mathcal{B}}_{c,n}}%
{\arg\max}\left\{  w_{n}(B)-D_n(B)\right\}  ,Z(\mathcal{B}^*_c)\right)
>\delta\right)  +\delta
\end{align*}
\begin{align*}
& \leq\mathbb{P}\left(  \sup_{B\in\hat{\mathcal{B}}_{c,n}:d(B,Z(\mathcal{B}^*_c)>\delta}\left\{  w_{n}(B)-D_n(B)\right\}  \geq 
\hat m_{c,n}\right)  +\delta\\
& \leq \mathbb{P}\left(  \sup_{B\in\hat{\mathcal{B}}_{c,n}:d(B,Z(\mathcal{B}^*_c))>\delta}\left\{  w_{n}(B)-D(B)\right\}  \geq 
\hat m_{c,n}-\varepsilon\right)  +\delta\\
& \leq \mathbb{P}\left(  \sup_{B\in\mathcal{B}_{c,n}^{p,+}:d(B,Z(\mathcal{B}^*_c))>\delta}\left\{  w_{n}(B)-D(B)\right\}  \geq 
\hat m_{c,n}-\varepsilon\right)  +2\delta
\end{align*}
which is by (\ref{mcmh})
$$ \leq \mathbb{P}\left(  \sup_{B\in\mathcal{B}_{c,n}^{p,+}:d(B,Z(\mathcal{B}^*_c))>\delta}\left\{  w_{n}(B)-D(B)\right\}  \geq 
 m_{c,n}^p-2\varepsilon\right)  +3\delta
$$
which is by (\ref{mmm})
\[
\leq \mathbb{P}\left(  \sup_{B\in\mathcal{B}_{c,n}^{p,+}:d(B,Z(\mathcal{B}^*_c))>\delta}\left\{  w_{n}(B)-D(B)\right\}  \geq  m_c^*-3\varepsilon
\right)  +4\delta%
\]
which  by (\ref{wnproba}), Lemma \ref{driftincrement}, and 
(\ref{Bp+_compact}), is in turn
\[
\leq \mathbb{P}\left(  \sup_{B\in\mathcal{B}^*_{c}:d(B,Z(\mathcal{B}^*_c))>\delta/2}\left\{  w_{n}(B)-D(B)\right\}  \geq  m_c^*-4\varepsilon
\right)  +5\delta%
\]
and this is by (\ref{wnapprox}) and  then by again Lemma 2.6 in \cite{KP90}
\[
\leq \mathbb{P}\left(  \max_{B\in\mathcal{B}^*_{c}:d(B,Z(\mathcal{B}^*_c))>\delta/2}\left\{ \sqrt{\lambda} W(B)-D(B)\right\}  \geq  m_c^*-5\varepsilon
\right)  +6\delta\leq 7\delta,
\]
provided $\varepsilon$ is chosen small enough.
The last two paragraphs of the proof of Theorem \ref{thL1} now yield the stated results. 
$\hfill\Box$

\color{black}
\bibliographystyle{imsart-nameyear}
\bibliography{Level_Set}
\end{document}